\documentclass[a4paper]{amsart}
\usepackage{amssymb}
\usepackage{amsmath}
\usepackage{amsthm}
\usepackage{bm}
\usepackage{graphicx}
\usepackage{ascmac}
\usepackage{comment}
\usepackage{mathtools}
\usepackage{mathrsfs}
\usepackage{tikz}
\usetikzlibrary{positioning}
\usetikzlibrary{cd}
\usepackage{pgfplots}
\usepackage[all]{xy}
\usepackage{latexsym}
\usepackage{amsfonts}

\theoremstyle{definition}
\newtheorem{defn}{Definition}
\newtheorem{thm}[defn]{Theorem}
\newtheorem{prop}[defn]{Proposition}
\newtheorem{lem}[defn]{Lemma}

\newtheorem{rem}[defn]{Remark}
\newtheorem{cf}[defn]{Example}

\newcommand{\beq}{\begin{eqnarray*}}
\newcommand{\eeq}{\end{eqnarray*}}

\newcommand{\rinj}{\rightarrowtail}

\newcommand{\rsurj}{\twoheadrightarrow}
\newcommand{\os}{\overset}
\newcommand{\ra}{\rightarrow}
\newcommand{\la}{\leftarrow}
\newcommand{\ol}{\overline}

\newcommand{\sSet}{\mathrm{Set}^{\Delta^{\mathrm{op}}}}

\newcommand{\C}{\mathscr{C}}
\newcommand{\D}{\mathscr{D}}
\newcommand{\E}{\mathscr{E}}
\newcommand{\Func}{\mathrm{Func}}
\newcommand{\Obj}{\mathrm{Obj}}
\newcommand{\intg}{\mathbb{Z}}

\newcommand{\nd}{\mathrm{nd}}

\newcommand{\Id}{\mathrm{Id}}

\newcommand{\colim}{\mathrm{colim}}
\newcommand{\Ar}{\mathrm{Ar}}
\newcommand{\Gap}{\mathrm{Gap}}

\newcommand{\real}{\mathbb{R}}
\newcommand{\Ch}{\mathrm{Ch}}

\title{Cofinality Theorems of \(\infty\)-Categories and Algebraic \(K\)-Theory}
\author{Hisato Matsukawa}
\date{} 

\begin{document}

\maketitle

\section*{Abstruct}
In this paper, we establish a theorem that proves a condition when an inclusion morphism between simplicial sets becomes a weak homotopy equivalence. Additionally, we present two applications of this result. The first application demonstrates that cofinal full inclusion functors of \(\infty\)-categories are weak homotopy equivalences. For our second application, we provide an alternative proof of Barwick's cofinality theorem of algebraic \(K\)-theory.

\section*{Acknowledgement}
I am deeply grateful to Prof.~Seidai~Yasuda for his guidance throughout the writing of this paper.

\tableofcontents

\section{Introduction}
It is a classical problem to investigate when an inclusion morphism \(A\to B\) between simplicial sets becomes a weak homotopy equivalence.
The main objective of this paper is to establish the theorem that gives this condition (Theorem \ref{the_main}).
To prove that, Theorem \ref{thm:finiteacyclification} plays crucial roles.
After that, we introduce two applications.
The first application is Theorem \ref{thm:app1} which states that cofinal full inclusion functors of \(\infty\)-categories are weak homotopy equivalence.
The second application is Theorem \ref{thm:TheCofinality_Wald}, providing another proof of Barwick's cofinality theorem of algebraic \(K\)-theory.
He used the additivity theorem to prove that, but we do not need it.

In section 2, we introduce a graph naturally associated with a simplicial set and investigate its properties, particularly its acyclicity.
For any finite simplicial set \(X\), there exists a simplicial set \(\widetilde{X}\) weak homotopy equivalent to \(X\), and the associated graph with \(\widetilde{X}\) is acyclic (see Theorem \ref{thm:finiteacyclification}).
This theorem plays a central role in proving Barwick's cofinality theorem.
When proving weak homotopy equivalence between a simplicial set \(Z\) and its subsimplicial set \(Y\), it is convenient to restrict ourselves to a finite subsimplicial set \(X\) of \(Z\), and construct homotopy from \(X\subset Z\) to \(X\to Y\subset Z\).
The cycles of the associated graph can pose obstacles to constructing this homotopy, but the Theorem \ref{thm:finiteacyclification} eliminates this hindrance.

In section 3, we prove our main theorem (Theorem \ref{the_main}).

In section 4, we introduce two applications.
The first application is Theorem \ref{thm:app1}, which states a cofinal full inclusion functor \(\D\to \C\) of \(\infty\)-categories induces a homotopy equivalence between their geometric realizations.
As the second application, We give another proof of Barwick's cofinality theorem of algebraic \(K\)-theory.
He used the additivity theorem to decompose the path space, but we do not need to do that.
Strictly, our statement is for labeled Waldhausen categories, not for Waldhausen categories.
Hence, our statement is a bit stronger than Barwick's.

Cofinality of \(K\)-theory is a classical problem with numerous known results.
We show the relations between them in Appendix \ref{append}.
In particular, Barwick's cofinality theorem is stronger than other many cofinality theorems.


\subsection{Notations}
We clarify some notations.
In this paper, the term \(\infty\)-category refers to quasi-category.
All sets, categories, and \(\infty\)-categories we consider are assumed to be small.
The symbol \(|X|\) denotes the geometric realization of a simplicial set \(X\), and \(N\C\) denotes the nerve simplicial set of a category \(\C\).
Occasionally, we may conflate a simplex \(a\in X_n\) of a simplicial set \(X\) with a morphism \(\Delta^n\to X\) corresponding to \(a\), and denote the morphism by \(a\) for convenience.

\section{Graphs associated with simplicial sets and their acyclicity}
Naturally, there is a graph associated with a simplicial set.
In this section, we investigate the properties of this graph, particularly focusing on its acylicity.

The purpose of this section is to prove Theorem \ref{thm:finiteacyclification}.
We consider operating on simplicial sets.
If the directed graph associated with the simplicial set is acyclic, we can carry out the operation inductively.
This theorem is used to prove cofinality theorems later.

First, let us recall the definition of a graph.
A directed graph \( G \) is a tuple \( G = (V, E, s, t) \), where \( V, E \) are sets, and \( s, t : E \to V \) are maps from \( E \) to \( V \).
An element of \( V \) is called a vertex, and an element of \( E \) is called an edge of \(G\).
The vertices \(s(e),t(e)\) are called the source and target of an edge \(e\in E\), respectively.
Hereafter, we will refer to a directed graph simply as a graph.

Let \( G = (V, E, s, t) \) be a graph.
A trail of \( G \) is a sequence \( (e_1, \ldots, e_n) \  (e_i \in E,\  1 \leq i \leq n) \) that has a length greater than or equal to 1, satisfying the condition \( t(e_i) = s(e_{i+1}) \) for all \( 1 \leq i \leq n-1 \).
Trails of length 0 are not considered in this paper.
A trail \( (e_1, \ldots, e_n) \) is said to be closed if \( s(e_1) = t(e_n) \). If \( G \) does not have any closed trails, \( G \) is said to be acyclic.

Let \( X \) be a simplicial set.
Let \( X_n^\nd \subset X_n \) be the set of non-degenerate \( n \)-simplices of \( X \), and let \( X^\nd = \coprod_n X_n^\nd \) be the set of all non-degenerate simplices of \( X \).

\begin{defn}
Let \( X \) be a simplicial set.
We set \( V_X = X_0 \), \( E_X = X_1^\nd \), \( s = d_1 \), and \( t = d_0 \).
Then, \( G_X = (V_X, E_X, s, t) \) is a graph.
We call \( G_X \) the graph associated with \( X \).
\end{defn}

\begin{defn}
We say a simplicial set \( X \) is acyclic if the graph \( G_X \) associated with \( X \) is acyclic. 
\end{defn}
This definition is not related to homotopy groups, homology groups, or similar concepts.
For any finite simplicial set \( X \), there exists an acyclic simplicial set weak homotopy equivalent to \( X \) (Theorem \ref{thm:finiteacyclification}).

\begin{cf}
  \begin{itemize}
    \item[(1)] The simplicial set \(\Delta^0\) is acyclic. The graph associated with \(\Delta^0\) has one vertex and no edges. We denote this graph by \(G_1\).
    \item[(2)] The simplicial set \(\Delta^1/\partial\Delta^1\) is not acyclic. The graph associated with this simplicial set has one vertex and a unique loop edge. This graph is a terminal object in the category of graphs.
    \item[(3)] For any \(n\ge 2\), the graph associated with \(\Delta^n/\partial\Delta^n\) is \(G_1\). Hence it is acyclic.
    \item[(4)] The graph associated with a simplicial set \(X\) is connected as an undirected graph if and only if \(X\) is connected.
    \item[(5)] The graph associated with a non-empty connected Kan complex is \(G_1\) or not acyclic. 
    \item[(6)] Let \(X\) be the nerve simplicial set of a partially ordered set \(S\). Then the graph associated with \(X\) can be written as follows. The set of vertices is the set of objects of \(S\). There is an edge from \(s\) to \(t\) if and only if \(s<t\) in \(S\). Moreover, there exists a trail (of length greater than or equal to \(1\)) from \(s\) to \(t\) if and only if \(s<t\). Hence \(X\) is acyclic.
    \item[(7)] Let \(X\) be the nerve simplicial set of a preordered set \(S\). Then \(G_X\) is acyclic if and only if \(S\) satisfies the antisymmetry axiom.
  \end{itemize}
\end{cf}

The following lemma is fundamental.

\begin{prop}\label{prop:acyclic}
\begin{itemize}
\item[(1)] The simplicial set \( \Delta^n \ (n \geq 0) \) is acyclic.
\item[(2)] Let \( X \) be an acyclic simplicial set, and let \( Y \subset X \) be a subsimplicial set of \( X \). Then, \( Y \) is also acyclic. 
\item[(3)] Let \( \{X_i\}_{i \in I} \) be a family of acyclic simplicial sets.
Then, the product \( \prod_{i \in I} X_i \) is also acyclic.
Moreover, with (2), we conclude that acyclic simplicial sets are closed under limits in the category of simplicial sets \( sSet \).
\item[(4)] Under the same condition as (3), the coproduct \( \coprod_{i \in I} X_i \) is also acyclic.
\item[(5)] Let \( X, Y, Z \) be simplicial sets, with \( Y \) and \( Z \) being acyclic.
If there exists a diagram of the form \( Y \la X \ra Z \), then \( Y \coprod_{X \times \{0\}} X \times \Delta^1 \coprod_{X \times \{1\}} Z \) is acyclic.
\item[(6)] Let \( X = \colim_{i \in I} X_i \) be a filtered colimit of simplicial sets.
If \( X_i \)'s are acyclic, then so is \( X \).
\end{itemize}
\end{prop}
\begin{proof}
(1)--(4) Obvious.
We need to exercise a little care with (3).
Generally, products do not commute with taking the associated graph, i.e., the equality \( G_{\prod_{i \in I} X_i} = \prod_{i \in I} G_{X_i} \) does not hold in general cases since morphisms in the category of graphs do not allow edges to be degenerate.

(5) The set of vertices of \( K = Y \coprod_{X \times \{0\}} (X \times \Delta^1) \coprod_{X \times \{1\}} Z \) is of the form \( K_0 = Y_0 \coprod Z_0 \).
There are no edges from a vertex of \( Z \) to a vertex of \( Y \) in \( K \).
We can reduce our statement to showing that the subsimplicial sets of \( K \) spanned by the set of vertices of \( Y \) and those spanned by the set of vertices of \( Z \) are acyclic, which are precise \( Y \) and \( Z \), respectively.

(6) If \( X \) is not acyclic, then there exists a closed trail \( T = (e_1, \ldots, e_n) \) in \( G_X \).
This closed trail \( T \) lifts to \( X_i \) as a closed trail for some \( i \in I \), which contradicts the hypothesis that \( X_i \) is acyclic.
\end{proof}

\begin{rem}
We remark on the definitions of the subsimplicial set spanned by a subset of the set of vertices.
Let \(K\) be a simplicial set, and \(S\) be a subset of the set of its vertices.
The subsimplicial set spanned by \(S\) in \(K\) is the subsimplicial set \(K_S\subset K\) that is a simplex \(s\in K_n\) in \(K\) is a simplex in \(K_S\) if and only if its all vertices contained in \(S\).
\end{rem}

A class of acyclic simplicial sets is closed under limits in \(\sSet\) but is not under colimits.
We give a counterexample.

\begin{cf}
We set \(X=\partial \Delta^1, Y=Z=\Delta^1\).
We define a morphism \(X\to Y\) by mapping \(0\mapsto 0,1\mapsto 1\), and  \(X\to Z\) by mapping \(0\mapsto 1,1\mapsto 0\).
The simplicial sets \(X,Y,Z\) are acyclic but \(Y\coprod_{X}Z\) is not.
\[
\xymatrix{
0\ar@/^4pt/[r]&1\ar@/^5pt/[l]
}
\]
\end{cf}

Let \(f:X\to Y\) be a morphism of simplicial sets, and \(a\in Y_n\) be an \(n\)-simplex of \(Y\).
We denote the pull back of \(X\) by \(a:\Delta^n\to Y\) via \(f/(a,n)=X\prod_{Y,a}\Delta^n\).

\begin{lem}
Let \(f:X\to Y\) be a morphism of simplicial sets, and \(d\ge 0\) be a natural number.
The following conditions are equivalent.
\begin{itemize}
\item[(1)] The fiber \(f/(a,n)\) is contractible for any natural number \(n\le d\) and non-degenerate \(n\)-simplex \(a\in Y_n\).
\item[(2)] The fiber \(f/(a\circ s,m)\) is contractible for any natural number \(n\le d\),\ non-degenerate \(n\)-simplex \(a\in Y_n\),\ natural number \(m\ge n\), and surjection \(s:\Delta^m\to \Delta^n\).
\item[(3)] The base change morphism \(X\prod_{Y}Z\to Z\) is weak homotopy equivalence for any simplicial set \(Z\) of dimension less than or equal to \(d\) and any morphism \(Z\to Y\).
\end{itemize}
\end{lem}
\begin{proof}
\((2)\Rightarrow (1),(3)\Rightarrow (1)\) Obvious.

\((2)\Rightarrow (3)\) Follows from Waldhausen's theorem (\cite[Lemma 1.4.B.]{Wa}).

\((1)\Rightarrow (2)\) The fibers \(|f/(a,n)|\) and \(|f/(a\circ s,m)|\) are homotopic.
\end{proof}

\begin{thm}\label{thm:finiteacyclification}
Let \(X\) be a finite simplicial set.
Then there exists an acyclic finite simplicial set \(\widetilde{X}\) and a morphism \(\pi:\widetilde{X}\to X\) satisfying the following condition; The fiber \(\pi/(a,n)\) is contractible for any \(n\ge 0\) and \(a\in X_n\).

Especially, any base change of \(\pi\) is a weak homotopy equivalence.
\end{thm}
\begin{proof}
First, we provide a rough sketch of the proof.
We begin by decomposing \(X\) into a disjoint union of its non-degenerate simplicies.
When we patch these simplices together in a natural way, we get \(X\) but it is not acyclic.
Therefore, we patch them together with redundancy.
 We obtain a suitable \(\widetilde{X}\) by eliminating the redundancy, which will be weak homotopy equivalent to \(X\).
 
 We give a total order on the set \(X^{\nd}\) of non-generate simplices in \(X\) such that \(x<y\) if \(\dim x< \dim y\).
We construct \(\widetilde{X}\) inductively on \(X^{\nd}\coprod\{\infty\}\) from larger to smaller.
We have the following information as induction hypotheses; simplicial set \(K_x\), morphism \(\pi_x:K_x\to X\), and \(\phi_{x,x'}:K_{x'}\to K_x\) for \(x'>x\ (x'\in X^\nd\coprod\{\infty\})\).
We also have the following conditions as induction hypotheses;
\begin{itemize}
\item[(1)] The simplicial set \(K_x\) is finite and acyclic.
\item[(2)] The simplicial set \(\pi_x/(x,n)\coprod_{K_x\prod_{X,x}\partial\Delta^n}\partial\Delta^n\) is contractible, where \(n\) is the dimension of \(x\).
\item[(3)] The equality \(\pi_x\circ \phi_{x,x'}=\pi_{x'}\) holds for all \(x'>x\ (x'\in X^\nd\coprod\{\infty\})\).
\item[(4)] The equality \(\phi_{x,x'}\circ \phi_{x',x''}=\phi_{x,x''}\) holds for all \(x''>x'>x\ (x',x''\in X^\nd\coprod\{\infty\})\).
\item[(5)] The morphism \(\pi_{x'}/(x',m)\coprod_{K_{x'}\prod_{X,x'}\partial\Delta^m}\partial\Delta^m \to \pi_x/(x',m)\coprod_{K_x\prod_{X,x'}\partial\Delta^m}\partial\Delta^m\) induced by \(\phi_{x,x'}\) is an isomorphism, where \(m\) is the dimension of \(x'\) for all \(x'>x\ (x'\in X^\nd)\).
\end{itemize}

In the case \(x=\infty\), we define \(K_\infty\) as \(\coprod_{x\in X'}\Delta^{\dim x}\), and \(\pi_\infty\) as the natural morphism \(K_\infty\os{(x)_{x\in X'}}{\to}Z \).
It is evident that \(K_\infty\) and \(\pi_\infty\) satisfy the above conditions (1)--(5).

Next, we consider the case where \(x\in X^\nd\).
Let \(x'\in X^\nd\coprod \{\infty\}\) be the minimum element among the elements greater than \(x\).
Set \(n=\dim x\).
Define \(K_{x}\) by the pushout diagram below.
\[
\xymatrix{
\pi_{x'}/(x,n)\ar[r]\ar[d]\ar@{}[rd]|{\mathrm{PO}}&\pi_{x'}/(x,n)\times\Delta^1\coprod_{\pi_{x'}/(x,n)\times\{1\}}\Delta^n\ar[r]\ar[d]&\Delta^{n}\ar[d]^{x}\\
K_{x'}\ar[r]^{\phi_{x,x'}}\ar@/_15pt/[rr]_{\pi_{x'}}&K_x\ar@{..>}[r]^{\exists!\pi_x}&X
}
\]
The outer rectangle is a Cartesian diagram.
The top horizontal arrow on the left is a map to \(\pi_{x'}/(x,n)\times\{0\}\).
We define \(\pi_x:K_x\to X\) as the morphism that makes the above diagram commutes.
Define \(\phi_{x,x''}\) by \(\phi_{x,x'}\circ \phi_{x',x''}\) for \(x''\in X^\nd\coprod \{\infty\}\) such that \(x''>x'\).

We show that \(K_x,\pi_x,\phi_{x,x''}\ (x''>x)\) satisfy conditions (1)--(5).
Condition (1) follows from Proposition \ref{prop:acyclic}.
Conditions (3) and (4) are clear from the construction.
Let \(L=\pi_{x'}/(x,n)\times\Delta^1\coprod_{\pi_{x'}/(x,n)\times\{1\}}\Delta^n\).
We will show condition (2).
\(L\coprod_{L\prod_{\Delta^n}\partial\Delta^n}\partial\Delta^n\) is homotopic to \(\Delta^n\),\ so it is contractible.
We show that the natural morphism \(f:L\coprod_{L\prod_{\Delta^n}\partial\Delta^n}\partial\Delta^n\to \pi_x/(x,n)\coprod_{K_x\prod_{X,x}\partial\Delta^n}\partial\Delta^n\) is an isomorphism.
It is sufficient to show that the geometric realization \(|f|\) of \(f\) is bijective because a morphism of simplicial sets is an isomorphism if and only if its geometric realization is bijective.
The morphism \(f\) is a morphism on \(\Delta^n\).
We will see that \(f\) is bijective on the fiber above any point of \(|\Delta^n|\).
On the boundary of \(\Delta^n\),\ the morphism is nothing but the identity map on \(|\partial\Delta^n|\), so it is bijective.
Let \(p'\in |\Delta^n|\) be an inner point, and \(p\) be its image in \(|X|\).
\[
\xymatrix{
|\pi_{x'}/(x,n)|\prod_{|X|}\{p\}\ar[r]\ar[d]\ar@{}[rd]|{\mathrm{PO}}&|L|\prod_{|X|}\{p\}\ar[r]\ar[d]&|\Delta^n|\prod_{|X|}\{p\}=\{p'\}\ar[d]\\
|K_{x'}|\prod_{|X|}\{p\}\ar[r]&|K_x|\prod_{|X|}\{p\}\ar[r]&\{p\}
}
\]
This diagram is a pullback of the previous one by \(\{p\}\to |X|\).
Since \(x\) is non-degenerate, the right vertical arrow is bijective.
The outer rectangle is a Cartesian diagram, implying that the left vertical arrow is also bijective.
Furthermore, the left rectangle is a pullback of pushout along an injection, so it is also a pushout diagram.
Thus, we conclude that the middle vertical arrow is bijective.
The following isomorphisms show that \(|f|\) is bijective on the fiber at \(p'\).
\[|L|\prod_{|X|}\{p\}\cong|L|\prod_{|\Delta^n|}\{p'\}\cong|L\coprod_{L\prod_{\Delta^n}\partial\Delta^n}\partial\Delta^n|\prod_{|\Delta^n|}\{p'\}\]
\[
|K_x|\prod_{|X|}\{p\}\cong|\pi_x/(x,n)|\prod_{|\Delta^n|}\{p'\}\cong|\pi_x/(x,n)\coprod_{K_x\prod_{X,x}\partial\Delta^n}\partial\Delta^n|\prod_{|\Delta^n|}\{p'\}
\]
To establish the condition (5), let us consider any \(x''>x\ (x''\in X^\nd)\).
The dimension \(m\) of \(x''\) is greater than or equal to \(n\).
An inner point \(p\in |\Delta^m|\) of \(|x''|\) does not intersect \(|x|\), meaning that \(\{p\}\prod_{|X|,|x|}|\Delta^n|=\emptyset\), because \(x''\) is non-degenerate.
Since \(|K_{x}|\) and \(|K_{x'}|\) are the same except on the fiber of points in \(|x|\), the map \(|\pi_{x'}/(x'',m)|\to |\pi_{x}/(x'',m)|\) is bijective on the fiber above inner point of \(x''\).
The rest follows from the same reasoning as in condition (2).

Let \(x_0\) be the minimum element of \(X^\nd\coprod\{\infty\}\).
We take \(\widetilde{X}=K_{x_0},\pi=\pi_{x_0}\).
We see that these satisfy the conditions of this proposition.
Take \(n\ge 0\) and non-degenerate \(a\in X_n\).
We have to see that \(\pi/(a,n)\) is contractible.
We show this by induction on \(n\) from smaller to larger.
First, if \(n=0\), we have the following isomorphisms since \(\partial\Delta^n=\emptyset\).
\[\pi/(x,n)\cong \pi/(x,n)\coprod_{\widetilde{X}\prod_{X,x}\partial\Delta^n}\partial\Delta^n\cong \pi_{x}/(x,n)\coprod_{K_{x}\prod_{X,x}\partial\Delta^n}\partial\Delta^n\]
The right-hand side is contractible, and so is \(\pi/(x,n)\).
In the case \(n\ge 1\), \(\widetilde{X}\prod_{X,x}\partial \Delta^n\to \partial\Delta^n\) is weak homotopy equivalence by the induction hypothesis.
\[
\pi/(x,n)\os{\sim}{\to} \pi/(x,n)\coprod_{\widetilde{X}\prod_{X,x}\partial\Delta^n}\partial\Delta^n\cong \pi_{x}/(x,n)\coprod_{K_{x}\prod_{X,x}\partial\Delta^n}\partial\Delta^n
\]
The morphism on the left-hand side is a weak homotopy equivalence.
The right-hand side is contractible, and so is \(\pi/(x,n)\).
\end{proof}

\section{Main theorem}
We prove our main theorem (Theorem \ref{the_main}) in this section.


A pair of simplicial sets \((X,Y)\) denotes a pair consisting of a simplicial \(X\) and its subsimplicial set \(Y\).

\begin{lem}\label{lem:mainlem}
Let \(X\) be a simplicial set, and \(Y\subset X\) be a subsimplicial set.
For any finite subsimplicial set \(X'\subset X\), there exists a pair of simplicial sets \((\widetilde{X},\widetilde{Y})\), a morphism of pairs of simplicial sets \(\pi:(\widetilde{X},\widetilde{Y})\to (X',X'\prod_{X} Y)\) and a homotopy \(H:(|\widetilde{X}\times\Delta^1|,|\widetilde{Y}\times\Delta^1|)\to (|X|,|Y|)\) from \(|\pi|\) to \((\widetilde{X},\widetilde{Y})\to (|Y|,|Y|)\subset (|X|,|Y|)\) such that \(\widetilde{X}\to X',\ \widetilde{Y}\to X'\cap Y\) are weak homotopy equivalences.
Then the inclusion morphism \(Y\subset Z\) is a weak homotopy equivalence.
\end{lem}
\begin{proof}
Let \((K,L)\) be a pair of topological spaces of the form \(((S^n\times I)/(\{*\}\times I), S^n\vee S^n)\) or \((S^n,\{*\})\) for some \(n\ge0\), where \(I=[0,1]\subset\real\) is the unit interval.
Any morphism of pairs of topological spaces \(\phi:(K,L)\to (|X|,|Y|)\) factors through the geometric realization \((|X'|,|X'|\cap |Y|)\) of a finite subsimplicial set \(X'\subset X\), since \(K\) is compact.
Take \((\widetilde{X},\widetilde{Y})\), \(\pi\), and \(H\) as hypothesized.
The morphism \(\phi\) lifts to \((|\widetilde{X}|,|\widetilde{Y}|)\) up to homotopy.
The homotopy \(H\) shows that \(|\phi|\) factors through \((|Y|,|Y|)\) up to homotopy.
This shows that the inclusion \(Y\subset X\) is a weak homotopy equivalence.
\end{proof}

Let \(K\) be a simplicial set with patial ordering \(\le\) on \(K_0\).
We say a pair \((K,\le)\) is an ordered simplicial set if \(a<b\) for all non-degenerate edges from the vertex \(a\) to the vertex \(b\).
If the ordered set \((K_0,\le)\) is a totally ordered set, we say \((K,\le)\) is a totally ordered simplicial set.
A simplicial set \(A\) admits a structure of an ordered simplicial set if and only if it is acyclic.
Moreover, there is a minimum ordering \(\le_{\mathrm{min}}\) on an acyclic simplicial set \(A\) defined as follows:
\[a <_{\mathrm{min}} b \text{ if and only if there exists a trail from } a \text{ to } b.\]
For any ordered simplicial set \((K,\le)\) and any refinement ordering \(\le'\) on \(K_0\), \((K,\le')\) is also an ordered simplicial set.
We often omit the ordering and just write \(K\) instead of \((K,\le)\).

Let \(K\) be an ordered simplicial set, and let \(v\in K_0\) be its vertex.
We denote \(K_{<v}, K_{\le v}, K_{>v},K_{\ge v}\subset K\) as the subsimplicial sets spanned by the set of vertices \(\{u\in K_0\mid u<v\},\ \{u\in K_0\mid u\le v\},\ \{u\in K_0\mid u>v\}\), and \(\{u\in K_0\mid u\ge v\}\), respectively.

\begin{thm}\label{the_main}
  Let \((X,Y)\) be a pair of simplicial sets.
  If \((X,Y)\) satisfies the following property, then the inclusion morphism \(Y\subset X\) is a weak homotopy equivalence.
  
  Let \(\phi:(K,L)\to (X,Y)\) be any morphism of pairs, where \(K,L\) are finite totally ordered simplicial sets, and \(v\in K_0\) be a vertex.
  We assume that \((K_{<v},L_{<v})\to (X,Y)\) lifts to \((Y,Y)\), where \(L_{<v}\) is the inverse image of \(K_{<v}\). 
  \[
  \xymatrix{
    L_{<v}\ar[r]\ar[d]&K_{<v}\ar[d]\ar@{..>}[ld]_{\exists}\\
    Y\ar[r]&X
  }
  \]
  Then there exists a homotopy \(H:(K\times\Delta^1,L\times\Delta^1)\to (X,Y)\) from \(\phi\) to \(\psi:(K,L)\to (X,Y)\) such that the restriction of \(\psi\) to \((K_{\le v},L_{\le v})\) lifts to \((Y,Y)\).
\end{thm}
\begin{proof}
  The proof follows along the lines of Lemma \ref{lem:mainlem}.
  Let \(X'\subset X\) be a finite subsimplicial set.
  Take \(\widetilde{X}\to X'\) as Theorem \ref{thm:finiteacyclification}, and \(\widetilde{Y}\) as \(Y\times_{X} \widetilde{X}\).
  We have to make a homotopy \((|\widetilde{X}\times\Delta^1|,|\widetilde{Y}\times\Delta^1|)\to (|X|,|Y|)\) satisfying certain conditions, but the hypothesis allows us to make that inductively.
\end{proof}

There is the dual result.

\begin{thm}
  Let \((X,Y)\) be a pair of simplicial sets.
  If \((X,Y)\) satisfies the following property, then the inclusion morphism \(Y\subset X\) is a weak homotopy equivalence.
  
  Let \(\phi:(K,L)\to (X,Y)\) be any morphism of pairs, where \(K,L\) are finite totally ordered simplicial sets, and \(v\in K_0\) be a vertex.
  We assume that \((K_{>v},L_{>v})\to (X,Y)\) lifts to \((Y,Y)\), where \(L_{>v}\) is the inverse image of \(K_{>v}\). 
  \[
  \xymatrix{
    L_{>v}\ar[r]\ar[d]&K_{>v}\ar[d]\ar@{..>}[ld]_{\exists}\\
    Y\ar[r]&X
  }
  \]
  Then there exists a homotopy \(H:(K\times\Delta^1,L\times\Delta^1)\to (X,Y)\) from \(\psi:(K,L)\to (X,Y)\) to \(\phi\) such that the restriction of \(\psi\) to \((K_{\ge v},L_{\ge v})\) lifts to \((Y,Y)\).
\end{thm}

\begin{rem}
  These theorems are generalizations of the well-known fact: let \(X\) be a \(\infty\)-category and \(Y\) be a full subcategory of \(X\).
  If the natural inclusion functor \(i:Y\to X\) admits an adjoint, then \(i\) is a weak homotopy equivalence.
  We show this fact by using our theorems.
  
  We assume \(i\) admits a left adjoint \(L\).
  In the condition of the hypothesis of the theorem, there is a morphism \(\psi:(K,L)\to (X,Y)\) and a homotopy from \(\phi\) to \(\psi\) satisfying the hypothesis of the theorem.
  Specifically, \(\psi\) is given by the following.
  \[
  \psi(u)=\begin{cases}
  \phi(u)&u<v\\
  i\circ L(\phi(u))&u\ge v
  \end{cases}\ \ \ \text{for \(u\in K_0\)}
  \]
  The unit map makes \(\psi\) a morphism of simplicial sets and gives a homotopy from \(\phi\) to \(\psi\).
  
  If \(i\) admits a right adjoint \(R\), we have a similar result.
  The map \(\psi\) is given by following.
  \[
  \psi(u)=\begin{cases}
  i\circ R(\phi(u))&u\le v\\
  \phi(u)& u>v
  \end{cases}\ \ \ \text{for \(u\in K_0\)}
  \]
  The counit map makes \(\psi\) a morphism of simplicial sets and  gives a homotopy from \(\psi\) to \(\phi\).
\end{rem}

\section{Applications}
In this section, we introduce two applications.
The first application is Theorem \ref{thm:app1}, which states a cofinal full inclusion functor \(\D\to \C\) of \(\infty\)-category induces a homotopy equivalence between their geometric realizations.
As a second application, we provide another proof of Barwick's cofinality theorem of algebraic \(K\)-theory.
He used the additivity theorem to decompose the path space, but we do not need to do that.
Strictly speaking, our statement is for labeled Waldhausen categories, not for Waldhausen categories.
Hence, our statement is a bit stronger than Barwick's.

\subsection{Cofinal inclusions of \(\infty\)-categories}

Let \(\C\) be a \(\infty\)-category and \(\D\) be its subcategory.
The subcategory \(\D\) is called cofinal in \(\C\) if for any \(C\in \C\), there exists \(C'\in \C\) such that \(C\coprod C'\) is in the essential image of \(\D\).

\begin{thm}\label{thm:app1}
  Let \(\C\) be a \(\infty\)-category and \(\D\) be its cofinal full subcategory.
  Then the natural inclusion \(\D\subset \C\) is a weak homotopy equivalence.
\end{thm}
\begin{proof}
  We may assume that \(\D\) is closed under equivalences since equivalences of categories do not change the homotopy types.
  We have to show that the inclusion \(\D\subset \C\) satisfies the hypothesis of Theorem \ref{the_main}.
  For any \(K,L, v\) and \(\phi\), we define \(\psi:K\to \D\) as follows:
  \[
  \psi(u)=
  \begin{cases}
    \phi(u)& u< v\\
    \phi(u)\coprod C'&u= v\\
    \phi(u)\coprod C'\coprod C'' &u>v
  \end{cases}\ \ \ \text{for all \(u\in K_0\)},
  \]
  where \(C',C''\) are objects in \(\C\) such that \(\phi(v)\coprod C'\in \D\) and \(C'\coprod C''\in\D\) (we may take \(C''=\phi(u)\)).
  The natural transforms give a homotopy from \(\phi\) to \(\psi\).
\end{proof}

\begin{thm}[dual]
  Let \(\C\) be a \(\infty\)-category and \(\D\) be its full subcategory.
  We assume that for any \(C\in\C\), there exists \(C'\in\C\) such that \(C\times C'\) is in the essential image of \(\D\).
  Then the natural inclusion \(\D\to \C\) is a weak homotopy equivalence.
\end{thm}

\subsection{Another proof of Barwick's cofinality theorem of algebraic \(K\)-theory}

In this section, we give another proof of Barwick's cofinality theorem of algebraic \(K\)-theory.

We recall the definition of Waldhausen \(\infty\)-category by Barwick (\cite[Definition 2.7]{Ba1}).
\begin{defn}
Let \(\C\) be a \(\infty\)-category and \(\C_\dagger\subset \C\) be its subcategory.
A morphism contained in \(\C_\dagger\) is called an ingressive morphism.
A pair \((\C,\C_{\dagger})\) is a Waldhausen \(\infty\)-category if it satisfies the following conditions.
\begin{itemize}
\item[(1)] The subcategory \(\C_\dagger\) contains the maximal Kan complex \(i\C\subset \C\).
\item[(2)] The \(\infty\)-category \(\C\) contains a zero object.
\item[(3)] Any morphism \(0\to X\) from a zero object \(0\) is ingressive.
\item[(4)] Pushout along ingressive morphisms exist.
\item[(5)] Pushout of ingressive morphisms are ingressive morphisms.
\end{itemize}
\end{defn}
We write \(\C\) instead of \((\C,\C_\dagger)\).

This definition is an infinity analog of ordinary category with cofibrations, not of ordinary Waldhausen category.

\begin{defn}
  Let \(\C,\D\) be Waldhausen \(\infty\)-categories, and \(F:\D\to \C\) be a functor of \(\infty\)-categories.
  The functor \(F\) is exact if \(F\) preserves zero objects, ingressives, and pushout along igressives.
\end{defn}

The definition of Waldhausen subcategory by Barwick is too strong.
We define it differently.
\begin{defn}
  Let \(\C,\D\) be Waldhausen \(\infty\)-categories, and \(\D\to \C\) be an inclusion of \(\infty\)-categories.
  We say \(\D\) is a Waldhausen subcategory of \(\C\) if the inclusion functor is exact, and a morphism in \(\D\) is an ingressive if and only if it is an ingressive in \(\C\) with cokernel in \(\D\).
\end{defn}

The \(\infty\)-category \(\Gap([n],\C)\) is defined in \cite[Definition 1.2.2.2]{HA}.
Let \(\Ar[n]\) be the category whose objects are \((i,j)\ (0\le i\le j\le n)\) and a set of morphisms from \((i,j)\) to \((i',j')\) is \(\{*\}\) if \(i\le i', j\le j'\), otherwise \(\emptyset\).
It is the arrow category of the ordered set \([n]=\{0<1<\cdots<n\}\).
The \(\infty\)-category \(\Gap([n],\C)\) is a subcategory of \(\Func(N(\Ar[n]),\C)\) spanned by functors \(A:N(\Ar[n])\to \C\) satisfying the following conditions.
\begin{itemize}
\item[(1)] The object \(A_{i,i}\ (0\le i\le n)\) is a zero object.
\item[(2)] The morphism \(A_{i,j}\to A_{i,k}\ (i\le j\le k)\) is an ingressive morphism.
\item[(3)] The diagram 
\[\xymatrix{
A_{i,j}\ar[r]\ar[d]&A_{i,k}\ar[d]\\
A_{j,j}\ar[r]&A_{j,k}
}\]
is a coCartesian diagram for all \(i\le j\le k\).
\end{itemize}
The family \((\Gap([n],\C))_{n}\) forms a simplicial \(\infty\)-category.

Let \(s_n\C\) be the set of objects of \(\Gap([n],\C)\).
The family \(s_*\C\) forms a simplicial set.

Let \(0\) be a fixed zero object in \(\C\).
There is a functor \(\C\to\Gap([1],\C)\) that maps \(X\) to 
\[
\xymatrix{
&0\\
0\ar[r]&X\ar[u]
}.
\]
Therefore, we get a map \(\Obj(\C)\to \pi_1(s_*\C,0):X\mapsto [X]\).

Let \(G\) be a group admitting a surjective group homomorphism \(\pi_1(s_*\C,0)\to G\).
We define a simplicial set \(s_*(\C,G)\) by 
\[s_n(\C,G)=\{(A,(g_0,\ldots,g_n))\in s_n\C\times G^n\mid g_{j}=g_i+[A_{i,j}]\text{ for all }i\le j\}.\]
Then \(s_*(\C,G)\) forms a simplicial set.
Forgetting the structure of \(G\), we get a morphism \(s_*(\C,G)\to s_*(\C)\) which makes \(s_*(\C,G)\) a connected covering space of \(s_*(\C)\).

Similarly, we define \(s_*(T,G)\) for \(T\) a set of objects of \(\C\).
Let \(T\) be a set of objects of \(\C\) and \(H\subset G\) be a subgroup generated by the image of elements of \(T\).
In the above setting, we define a subsimplicial set \(s_*(T,G)\subset s_*(\C,G)\) by
\[
s_n(T,G)=\{(A,(g_0,\ldots,g_n))\in s_n(\C,G)\mid A_{i,j}\in T, g_0\in H\text{ for all }i\le j\}.
\]
We write \(s_*(T)\) instead of \(s_*(T,0)\).
The simplicial set \(s_*(T,G)\) is a connected covering space of \(s_*(T)\).

\begin{thm}\label{thm:wald_main}
Let \(\C\) be a Waldhausen \(\infty\)-category, \(G\) be a group with a surjective group homomorphism \(\pi_1(s_*\C,0)\to G\), and \(T\) be a set of objects of \(\C\).
We assume that
\begin{itemize}
\item[(1)] The set of objects \(T\) is closed under taking finite coproducts. In particular, \(T\) contains a zero object.
\item[(2)] For any \(A\in \Obj(\C)\),\ there exists \(A'\in \Obj(\C)\) such that \(A\coprod A'\in T\).
\item[(3)] For any \(A\in \Obj(\C)\), if \([A]\in H\) then there exists \(A'\in T\) such that \(A\coprod A'\in T\).
\end{itemize}
Then the natural inclusion \(s_*(T,G)\to s_*(\C,G)\) is a weak homotopy equivalence.
\end{thm}
\begin{proof}
  We may assume that \(T\) is closed under equivalences in \(\C\) since equivalences induce homotopies in \(s_*(\C,G)\).
The proof goes along the lines of Theorem \ref{the_main}.
Take \(K,L,v\) and \(\phi\) as the hypothesis of Theorem \ref{the_main}.
We claim that there exist \(C',C''\in \C\) such that
\begin{itemize}
\item[(a)] \(g_0+[C']\in H\), where \((A,(g_0))=\phi(v)\).
\item[(b)] \(A_{0,1}\coprod C'\in T\) for any \(u<v\) and any edge \(e\in K_1\) from \(u\) to \(v\), where \((A,(g_0,g_1))=\phi(e)\).
\item[(c)] If \(v\in L\), then \(C',C''\in T\).
\item[(d)] \(C'\coprod C''\in T\).
\end{itemize}
Let \(C\in \C\) be an object satisfying condition (a).
For all edge \(e\in K_1\) from \(u<v\) to \(v\), take \(D_e\in T\) such that \(A_{0,1}\coprod C\coprod D_e\in T\), where \((A,(g_0,g_1))=\phi(e)\).
We set \(C'= C\coprod (\coprod_e D_e)\) and take \(C''\) such satisfies condition (d).
If \(v\in L\), we may will \(C=0, C''=0\).
These \(C',C''\) satisfies conditions (a)--(d).

We define \(\psi:K\to s_*(\C,G)\) by \(K_n\ni a\mapsto (B,(h_0,\ldots,h_n))\in s_n(\C,G)\), where
\beq
B_{k,l}&=&
\begin{cases}
  A_{k,l}  & w<v\\
  A_{k,l} \coprod C'&w=v\\
  A_{k,l} \coprod C'\coprod C''& u<v<w\\
  A_{k,l} \coprod  C''& v=u\\
  A_{k,l} & v<u,
\end{cases}
\text{ where } v_k(a)=u,v_l(a)=w,\\
h_k &=&
\begin{cases}
  g_k & u<v\\
  g_k+[C'] + [C']&u=v\\
  g_k+[C']+[C''] & v<u,
\end{cases}
\text{ where }v_k(a)=u,\\
\text{where } \phi(a)&=&(A,(g_0,\ldots,g_n)),\ \ \ v_k(a)\text{ is a k-th vertex of }a.
\eeq
Define a homotopy from \(\phi\) to \(\psi\) by \(h_m:K_n\to s_{n+1}(\C,T):\ a\mapsto (C,(i_0,\ldots,i_{n+1}))\ (m=0,\ldots,n)\), where
\beq
C_{k,l}&=&
\begin{cases}
A_{k,l} & (l\le m)\\
A_{k,l-1}&(k\le m<l,\  v_{l-1}(a)<v)\\
A_{k,l-1}\coprod C'&(k\le m<l,\  v_{l-1}(a)=v)\\
A_{k,l-1}\coprod C'\coprod C''&(k\le m<l,\  v<v_{l-1}(a))\\
B_{k-1,l-1}  & (m< k),
\end{cases}\\
i_k &=&
\begin{cases}
g_k & (k\le m)\\
h_{k-1}&(m<k),
\end{cases}\\
\text{where } \phi(a)&=&(A,(g_0,\ldots,g_n)), \ \ \ \psi(a)\ \ =\ \ (B,(h_0,\ldots,h_n)).
\eeq
To define a choice of coproducts, we fix a functor \(\C\times N(\Ar[2])\to \C\) such that
\[
(C,(i,j))\mapsto \begin{cases}
  C& (i=j)\\
  C\coprod C'& (i=0,j=1)\\
  C\coprod C'\coprod C'' & (i=0,j=2)\\
  C\coprod C''&(i=1,j=2)
\end{cases}
\]
and the restriction to \(\C\times\{0\}\to \C\) is the identity functor.
We may check that this homotopy satisfies the hypotheses of Lemma \ref{lem:mainlem}.
\end{proof}

Labeled Waldhausen \(\infty\)-category is defined by Barwick in \cite[Definition 9.2]{Ba1}.
It is an infinity analog of classical Waldhausen category, which is category with cofibrations and weak equivalences.
A Waldhausen \(\infty\)-category \(\C\) is naturally a labeled Waldhausen \(\infty\)-category, where weak equivalences are equivalences \(w\C=i\C\).

A functor of labeled Waldhausen \(\infty\)-categories is called labeled exact if it is exact as a functor of Waldhausen \(\infty\)-categories and it preserves weak equivalences.
Let \(\C,\D\) be labeled Waldhausen \(\infty\)-categories and assume \(\D\) is a Waldhausen subcategory of \(\C\).
The labeled Waldhausen \(\infty\)-category \(\D\) is called a labeled Waldhausen subcategory of \(\C\) if a morphism of \(\D\) is weak equivalence just in the case its image in \(\C\) is so.

Let \(\C=(\C,\C_\dagger,w\C)\) be a labeled Waldhausen \(\infty\)-category.
The \(\infty\)-category \(w\Gap([n],\C)\) is a subcategory of \(\Gap([n],\C)\) with same objects.
A morphism \(A\to B\) in \(\Gap([n],\C)\) is a morphism in \(w\Gap([n],\C)\) if \(A_{i,j}\to B_{i,j}\) is a weak equivalence for all \(i\le j\).
Here, \((w\Gap([n],\C))_n\) forms a simplicial \(\infty\)-category, and its \(K\)-theory space is defined as
\[
K(\C)=\Omega|(w\Gap([n],\C))_n|.
\]
The group
\[K_n(\C)=\pi_n(K(\C),0)\]
is the \(K\)-group of \(\C\) in degree \(n\).

\begin{thm}[Barwick's Cofinality Theorem, {\cite[Theorem 10.19.]{Ba1}}]\label{thm:TheCofinality_Wald}
  Let \(\C\) be a labeled Waldhausen \(\infty\)-category and \(\D\) be its labeled Waldhausen subcategory of \(\C\).
The map \(K(\D)_n\to K(\C)_n\) is an isomorphism for \(n\ge 1\) and is injective for \(n=0\) if the following conditions are satisfied.
\begin{itemize}
\item[(1)] For any \(C\in \C\), there exists \(C'\in \C\) such that \(C\coprod C'\) is in the essential image of \(\D\).
\item[(2)] For any \(C\in \C\), if \([C]\in K_0(\C)\) is in the image of \(K_0(\D)\), there exists \(C'\in \D\) such that \(C\coprod C'\) is in the essential image of \(\D\).
\item[(3)] For any labeled morphism \(C_0\to C_1\in w\C\) with \(C_i\in \D (i=0,1)\), there exists a labeled morphism \(C_0'\to C_1'\) in \(w\D\) such that \(C_0\coprod C_0'\to C_1\coprod C_1'\) is in the essential image of \(w\D\).
\item[(4)] For any \(A\in \Gap([2],\C)\), if each \(1\)-face of \(A\) is contained in the image of \(\Gap([1],\D)\), then \(A\) is in the image of \(\Gap([2],\D)\).
\end{itemize}
\end{thm}
\begin{proof}
The first part of the proof follows the idea in the proof of Waldhausen's cofinality theorem (\cite[Proposition 1.5.9.]{Wa}).
We will construct an infinity analog of \(A(m,w)\) and prove the statement degreewise.
Let \(w_m\C\) be the set of \(m\)-simplices of \(w\C\), and \(\E(\C,m)\) be a full subcategory of the \(\infty\)-category \(\Func(\Delta^m,\C)\) spanned by elements of \(w_m\C\).
Define a subcategory \(\E(\C,m)_\dagger\) of the \(\infty\)-category \(\E(\C,m)\) as follows.
Its set of objects is the same as that of \(\E(\C,m)\).
A morphism \((C_0\to\cdots\to C_m)\to (C_0'\to\cdots\to C_m')\) of \(\E(\C,m)\) is a morphism of \(\E(\C,m)_\dagger\) if \(C_i\to C_i'\) in \(\C_\dagger\) for all \(i\).
The pair \((\E(\C,m),\E(\C,m)_\dagger)\) forms a Waldhausen \(\infty\)-category.
The simplicial set \(s_*(E(\C,m))\) can be identified with \((w_m\Gap([n],\C))_n\).
The map of sets \(\Obj(E(\C,m))\to K_0(\C): (C_0\to \cdots \to C_m)\mapsto [C_0]\) induces a surjective homomorphism of groups \(\pi_1(s_*(E(\C,m)),0)\to K_0(\C)\).
The double simplicial set \((s_n(E(\C,m),K_0(\C)))_{n,m}\) is a covering space of \((w_m\Gap([n],\C))_{n,m}\).
Similarly for \(\D\).
Let \(H\) be the image of \(K_0(\D)\to K_0(\C)\).
The set \(\Obj(E(\D,m))\) can be identified with a set of objects of \(E(\C,m)\) and condition (4) asserts that \((s_n(E(\D,m),H))_{n,m}=(s_n(\Obj(E(\D,m)),K_0(\C)))_{n,m}\).
We will show that the natural morphism \((s_n(\Obj(E(\D,m)),K_0(\C)))_{n,m}\to (s_n(E(\C,m),K_0(\C)))_{n,m}\) is an isomorphism.
If it is shown, the statement of this theorem follows since \((s_n(E(\D,m),H))_{n,m}\) and \((s_n(E(\C,m),K_0(\C)))_{n,m}\) are covering space of \((s_n(E(\D,m)))_{n,m}\) and \((s_n(E(\C,m)))_{n,m}\), respectively.

Coproducts of \(E(\C,m)\) are coproducts of \(\C\) in each index.
So \(\Obj(E(\D,m))\) is closed under coproducts in \(E(\C,m)\).
Let \(C_0\to \cdots\to C_m\) be an object of \(E(\C,m)\).
Take \(C_0'\in\C\) as in (2).
There exists \(C_i'\in \Obj(E(\D,m))\) for each \(i\ge 1\) such that \(C_i\coprod C_0'\coprod C_i'\in \Obj(E(\D,m))\) since \([C_i\coprod C_0']=[C_i]+[C_0']=[C_0]+[C_0']\in H\).
Set \(C'=C_0'\coprod \cdots\coprod C_m'\).
For each \(i=0,\ldots,m-1\), take a morphism \(D_{i}\to D_{i}'\) in \(w\D\) such that \(C_{i}\coprod C_0'\coprod D_i \to C_{i+1}\coprod C_0'\coprod D_{i}'\) is in \(w\D\).
The coproduct
\beq
(C_0\to\cdots\to C_m)\coprod (C'\coprod D_0\coprod D_1\coprod\cdots \coprod D_m&\os{\Id}{\to}& C'\coprod D_0'\coprod D_1\coprod\cdots \coprod D_m\os{\Id}{\to}\\
&\cdots&\os{\Id}{\to} C'\coprod D_0'\coprod\cdots \coprod D_{m-1}'\coprod D_m)
\eeq
is in \(\Obj(E(\D,m))\).
If \([C_0]\in H\), we can take \(C_0\) as an object of \(\Obj(\D)\), then the right-hand side becomes an element of \(\Obj(E(\D,m))\).
Therefore, we can apply the Theorem \ref{thm:wald_main} to \(\Obj(E(\D,m))\).
\end{proof}

\begin{rem}
  There is no condition (2) in the original statement of Barwick's theorem.
  He used the additivity theorem and the Segal condition, and reduced it to Waldhausen's cofinality theorem.
  I think we need some conditions instead of Waldhausen's condition.
  If \(\D\) is closed under extension in \(\C\), there is no problem.
  By Grayson's trick, condition (1) automatically implies condition (2).
\end{rem}

\appendix
\section{Relations between other cofinality theorems}\label{append}
There are various types of cofinality theorems of algebraic \(K\)-theory.
Here, we introduce three classical cofinality theorems, Waldhausen's cofinality theorem, Thomason-Trobaugh cofinality theorem, and Grayson-Staffeldt cofinality theorem.
Lastly, we demonstrate that Barwick's cofinality theorem for labeled Waldhausen \(\infty\)-categories encompasses all the cofinality theorems previously discussed.

In Waldhausen's original paper, there were some mistakes in the proof of his cofinality theorem, and the statement was wrong.
We have corrected the statement.

\begin{defn}
  Let \(F:\C\to \D\) be a exact functor between Waldhausen \(\infty\)-categories.
  The functor \(F\) is said strictly cofinal if for any \(D\in \D\), there exists \(C,C'\in \C\) such that \(F(C)\cong D\coprod F(C')\).
\end{defn}

\begin{thm}[{Waldhausen, \cite[Proposition 1.5.9.]{Wa}}]\label{thm:WC}
Let \(\C\) be a classical Waldhausen category and \(\D\) be its Waldhausen subcategory.
We assume the following conditions.
\begin{itemize}
\item[(A)] The Waldhausen subcategory \(\D\) is closed under extension. This means that for any cofibration sequence \(C_0\rinj C_1\rsurj C_2\), if \(C_0,C_2\in \D\), then \(C_1\in \D\). Moreover, the morphisms \(C_0\to C_1,C_1\to C_2\) belong to \(\D\).
\item[(B)] The inclusion \(\D\subset \C\) is strictly cofinal.
\item[(C)] If \(C_0\to C_1\) is a weak equivalence in \(\C\), and \(C_0,C_1\in \D\), then there exists a weak equivalence \(C_0'\to C_1'\) in \(\D\) such that \(C_0\coprod C_0'\to C_1\coprod C_1'\) is a morphism in \(\D\).
\end{itemize}
Then the morphism \(K(\D)\to K(\C)\) is a weak homotopy equivalence.
\end{thm}
\begin{rem}
  For condition (A), simply having \(C_1\in \D\) is not sufficient.
  \cite{Za} provides a counterexample.
  
  Condition (C) is trivial if weak equivalences are isomorphisms in \(\C\), but it is crucial in general.
  If we lack this condition, the statement will be false.
  We provide a counterexample.
\end{rem}

\begin{cf}
  Let \(\mathrm{fin.Set}_*\) be the classical category of finite pointed sets.
  We define the category \(\C\) with the set of objects \(\Obj(\mathrm{fin.Set}_*)\times\Obj(\mathrm{fin.Set}_*)\), and morphisms from \((A,B)\) to \((A',B')\) are pairs \((f,g)\), where \(f:A\to A'\) and \(g:B\to A'\vee B'\) are maps of pointed sets.
  \(\C\) is a Waldhausen category with cofibrations as split injections, and a morphism \((f,g):(A,B)\to (A',B')\) is a weak equivalence if \((f\vee \{*\},g):A\vee B\to A'\vee B'\) is a bijection.
  Let \(\C'\subset \C\) be the subcategory whose set of objects are \(\Obj(\C')=\Obj(\C)\), and a morphism \((f,g):(A,B)\to (A',B')\) in \(\C\) is a morphism in \(\C'\) if \(g\) factors through \(B'\to A'\vee B'\).
  The subcategory \(\C'\) is a Waldhausen subcategory and satisfies all conditions of the above theorem but (C).
  The category \(\C'\) is categorically equivalent to \(\mathrm{fin.Set}_*\times \mathrm{fin.Set}_*\), hence \(K_0(\C')\cong\intg\oplus \intg\).
  But \(K_0(\C)\cong\intg\) since every object in \(\C\) is weak equivalent to a object of the form \((A,\{*\})\).
  Therefore, \(K_0(\C')\to K_0(\C)\) cannot be injective.
\end{cf}

\begin{rem}
  Ullman suggests adding a fullness condition to the original Waldhausen's statement (\cite[8.2.6.]{Ul}).
  Our statement is a bit stronger than that.
\end{rem}

\begin{thm}[Thomason-Trobaugh, {\cite[Theorem 1.10.1.]{TT}}]\label{thm:TTC}
Let \(\C\) be a classical Waldhausen category with a cylinder functor satisfying a cylinder axiom, and \(G\) be an abelian group admitting a surjective homomorphism of groups \(\pi:K_0(\C)\to G\).
Let \(\C^w\subset \C\) be a full subcategory spanned by the objects \(C\in \C\) such that \(\pi[C]=0\) in \(G\).
The subcategory \(\C^w\) be a Waldhausen subcategory of \(\C\).
Then there is a homotopy fiber sequence
\[
K(\C^w)\to K(\C)\to G,
\]
where \(G\) is a Eilenberg-MacLane space \(K(G,0)\cong \coprod_{G}\{*\}\).
\end{thm}
\begin{rem}
  Thomason and Trobaugh assert stronger cofinality in the same paper (\cite[Exercise 1.10.2.]{TT}).
  However, the conditions of their cofinality theorem are not sufficient.
  Ullmann corrected the statement by adding some conditions.
  He asserts that condition (c) of the theorem below is necessary and provides a counterexample in that case except for the cylinder functor (e) (\cite[Theorem 8.2.1.]{Ul}).
  He also asserts that the fullness condition is necessary (\cite[8.2.6.]{Ul}).
  We provide a counterexample (Example \ref{cf:Ullman2}) for that case.  
  That is a triangulated version of the example given in \cite{Za}.
\end{rem}

\begin{thm}[Waldhausen-Thomason-Ullman, {\cite[Theorem 8.2.1.]{Ul}}]\label{thm:UC}
  Let \(\C\) and \(\D\) be classical Waldhausen categories.
  Suppose \(\C\) is a full subcategory of \(\D\) satisfying the following conditions.
  \begin{itemize}
  \item[(a)] \(\C\) is a Waldhausen subcategory of \(\D\). Especially, a morphism in \(\C\) is a cofibration if and only if it is a cofibration in \(\D\) with cokernel in \(\C\).
  \item[(b)] A map in \(\C\) is a weak equivalence if and only if it is one in \(\D\).
  \item[(c)] Every object in \(\D\) which is weakly equivalent to an object in \(\C\) is itself in \(\C\).
  \item[(d)] If \(X\to Y\to Z\) is a cofiber sequence in \(\D\) and \(X,Z\) are in \(\C\), then \(Y\) is in \(\C\).
  \item[(e)] \(\D\) has mapping cylinders satisfying the cylinder axiom and \(\C\) is closed under them.
  \item[(f)] For every object \(X\) in \(\D\) there is an object \(\ol{X}\) in \(\D\) such that \(X\coprod \ol{X}\) is in \(\C\).
  \end{itemize}
  Then \[K(\C)\to K(\D)\to K_0(\D)/K_0(\C)\]
  is a homotopy fiber sequence of connective spectra.
\end{thm}

\begin{rem}
  This theorem is stronger than Theorem \ref{thm:TTC}.
\end{rem}

\begin{cf}[A triangulated version of \cite{Za}]\label{cf:Ullman2}
  Let \(\C\) be the category whose objects are \(\Obj(\Ch^b(\intg)\times\Ch^b(\intg))\), and morphisms from \((A,B)\) to \((A',B')\) are maps of chain complexes \(f:A\oplus B\to A'\oplus B'\).
  The category \(\C\) is equivalent to the category \(\Ch^b(\intg)\) via the functor \((A,B)\mapsto A\oplus B\).
  The category \(\Ch^b(\intg)\) has the standard Waldhausen structure and standard cylinder functor as mapping cylinder satisfying the cylinder axiom, so is \(\C\).
  The category \(\D=\Ch^b(\intg)\times\Ch^b(\intg)\) is a Waldhausen subcategory of \(\C\) closed under the cylinder functor.
  The pair \(\C,\D\) satisfies all conditions of Theorem \cite{Ul} except for the fullness.
  The \(K\)-groups are \(K_i(\D)\cong K_i(\intg)\oplus K_i(\intg),\ K_i(\C)\cong K_i(\intg)\).
  Hence the natural morphism \(K_i(\D)\to K_i(\C)\) cannot be isomorphisms in higher degrees.
\end{cf}

\begin{thm}[Grayson-Staffeldt, {\cite[Theorem 1.1]{LMA}, \cite[Theorem 6.1]{ESA}, \cite[Theorem 2.1.]{St}}]\label{thm:GSC}
Let \(\C\) be a classical exact category, and \(\D\) be its exact subcategory.
Suppose \(\D\) is cofinal and closed under extensions in \(\C\).
Then there is a homotopy fibration sequence
\[
K(\C)\to K(\C) \to K_0(\C)/K_0(\D).
\]
\end{thm}

We state that Theorem \ref{thm:TheCofinality_Wald} induces all the above theorems.
\begin{thm}
Theorem \ref{thm:TheCofinality_Wald} induces Theorems \ref{thm:WC}, \ref{thm:TTC}, \ref{thm:UC}, and \ref{thm:GSC}. 
\end{thm}
\begin{proof}
We remark that a classical Waldhausen category is naturally a labeled Waldhausen \(\infty\)-category.

(Theorem \ref{thm:WC})
Condition (B) of Theorem \ref{thm:WC} implies conditions (1) and (2) of Theorem \ref{thm:TheCofinality_Wald}.
(C) and (A) imply (3) and (4), respectively.

(Theorem \ref{thm:TTC})
For any \(C\in \C\), \(C\coprod\Sigma C\in \D\).
This implies (1).
(2) is trivial.
(3) is also trivial since \(\D\) is a full subcategory of \(\C\).
Since \(D\) is closed under extensions, (4) applied.

(Theorem \ref{thm:UC})
(1) is implied by (f).
(2) is implied by (d), fullness and Grayson's trick (\cite[Section 1]{LMA} or \cite[Exercise II.9.14]{We}).
(3) is implied by (b),(c) and fullness.
(4) is implied by (d) and fullness.

(Theorem \ref{thm:GSC})
An exact category is naturally a Waldhausen category.
Cofibrations are admissible injections, and weak equivalences are isomorphisms.
(1) follows from the cofinality.
(4) follows since \(\D\) is closed under extensions.
(3) is trivial because labeled morphisms are equivalences.
(2) follows from Grayson's trick.
\end{proof}

\vspace{0.1cm}

Department of Mathematics, Faculty of Science, Hokkaido University

Kita 10, Nishi 8, Kita-Ku, Sapporo, Hokkaido, 060-0810, Japan

Email: matsukawa.hisato.f4@elms.hokudai.ac.jp

\end{document}